\documentclass{article}
\usepackage{amssymb,amscd,verbatim, amsthm, amsmath, comment}
\theoremstyle{definition}
\newtheorem{theorem}{Theorem}[section]

\newtheorem{lemma}[theorem]{Lemma}
\newtheorem{corollary}[theorem]{Corollary}

\newcommand{\Z}{\mathbb{Z}}
\newcommand{\N}{\mathbb{N}}
\newcommand{\bi}{{\bf i}}
\newcommand{\bj}{{\bf j}}
\newcommand{\bk}{{\bf k}}

\title{Sums of Quaternion Squares and a Theorem of Watson}
\author{Tim Banks, Spencer Hamblen, Tim Sherwin, and Sal Wright}
\date{}

\begin{document}

\maketitle

\begin{abstract}
    We use a representability theorem of G. L. Watson to examine sums of squares in Quaternion rings with integer coefficients.  This allows us to determine a large family of such rings where every element expressible as the sum of squares can be written as the sum of 3 squares.
\end{abstract}

\section{Introduction}

\begin{theorem}[Waring's Problem/Hilbert-Waring Theorem]
For every integer $k \geq 2$ there exists a positive integer $g(k)$ such that every positive integer is the sum of at most $g(k)$ $k$-th powers of integers.
\end{theorem}

The idea behind Waring's Problem -- examining sums of powers -- can be easily extended to any ring.   For an excellent and thorough exposition of the research on Waring's Problem and its generalizations, see Vaughan and Wooley \cite{wooley}.  We will specifically be examining sums of squares in quaternion rings.

Let $LQ_{a,b}$ denote the quaternion ring
\[ \{\alpha_0 + \alpha_1 {\bf i} + \alpha_2 {\bf j} + \alpha_3 {\bf k} \mid \alpha_n,a,b \in {\mathbb Z}, {\bf i}^2 = -a, {\bf j}^2 = -b, {\bf i}{\bf j}=-{\bf j}{\bf i}={\bf k}\},\]
and let $LQ_{a,b}^n$ denote the additive group generated by all $n$th powers in $LQ_{a,b}$.  We can then examine the {\em Waring numbers} $g_{a,b} (k)$, the least positive integer such that every element of $LQ_{a,b}^n$ can be written as the sum of at most $g_{a,b} (k)$ $k$-th powers of elements of $LQ_{a,b}$.

Note here that ${\bf k}^2 = -ab$, and that if $a = b = 1$, we have the {\em Lipschitz quaternions}, where Pollack \cite{pollack} recently proved the analogue of the Hilbert-Waring Theorem -- that $g_{1,1}(k)$ exists for all $k \geq 2$.

The paper will examine sums of squares, and determine $g_{a,b} (2)$ for various $a,b > 0$.  Previous work of Cooke, Whitfield, and the second author \cite{chw} showed that $3 \leq g_{a,b}(2) \leq 5$ for all $a,b > 0$. We use a representability theorem of Watson (Theorem \ref{watsonthm}) to extend these results to the following theorem.

\begin{theorem} \label{main}
Given positive square-free integers $a$ and $b$, if $\gcd(a,b) \leq 2$, then $g_{a,b} (2) = 3$.
\end{theorem}

As a Corollary, using a result of Moree \cite{moree}, we get the following.

\begin{corollary} \label{densitythm}
Let $k$ and $g$ be positive integers, and let 
\[C(g,k) = \lim_{x \to \infty} \frac{|\{(a,b) \in \N \times \N \mid a,b \leq x, g_{a,b}(k) = g\}|}{x^2}\]
If $C(3,2)$ exists, then $C(3,2) \geq 0.322590$.
\end{corollary}

\section{Outline of methods, Theorem of Watson} \label{2}

Let $x=x_0+x_1\bi+x_2\bj+x_3\bk \in LQ_{a,b}$.  We call $x_0$ the {\em real} part of $x$ and $x_1\bi+x_2\bj+x_3\bk$ the {\em pure} part of $x$, with $x_1, x_2, x_3$ the {\em pure coefficients}.  Then note that 
\begin{equation} \label{al2}
	x^2=x_0^2-ax_1^2-bx_2^2 - abx_3^2 + 2x_0x_1\bi + 2x_0x_2\bj + 2x_0x_3\bk.
\end{equation}    

We then have that (see Equation (2) of Cooke et al. \cite{chw} or Theorem 1.1, part 2, of Pollack \cite{pollack})
\begin{equation} \label{sqform}
	LQ_{a,b}^2 = \{\alpha_0 + 2\alpha_1 \bi + 2\alpha_2 \bj + 2\alpha_3 \bk \mid \alpha_n \in \mathbb{Z}\}.
\end{equation}  

Our goal is, given $\alpha = \alpha_0+2\alpha_1\bi+2\alpha_2\bj+2\alpha_3\bk \in LQ_{a,b}^2$, to find $x,y,z \in LQ_{a,b}$ such that $\alpha = x^2 + y^2 + z^2$.  Our process will be overall similar to that used in Theorem \ref{modp2} of Cooke et al. $\cite{chw}$: we start by letting $z = 1 +\alpha_1\bi +\alpha_2\bj +\alpha_3\bk$, and note that $\alpha - z^2 \in \Z$.  To prove Theorem \ref{main}, it then suffices to prove that any integer can be written in the form $x^2 + y^2$ with $x,y \in LQ_{a,b}$.

If we stipulate that $x = x_0$ is an strictly real quaternion, and that $y = y_1\bi+y_2\bj+y_3\bk$ is a strictly pure quaternion, the problem then reduces to showing that the {\em indefinite} integral quadratic form $x_0^2 - ay_1^2 - by_2^2 - aby_3^2$ is universal (i.e., represents all integers).  

We first need some terminology.  If $f(x_0, y_1, y_2 ,y_3)$ is a (in our case, quaternary) quadratic form and $d \in \Z$, we say that $f$ {\em represents} $d$ if there exist $r_0, r_1, r_2, r_3 \in \Z$ such that $f(r_0, r_1, r_2, r_3) = d$.  If there exist $r_0, r_1, r_2, r_3 \in \Z$ such $f(r_0, r_1, r_2, r_3) = d$ and $\gcd(r_0, r_1, r_2, r_3) = 1$, we say that $f$ {\em properly represents} $d$.  For $n \in \N$, we can similarly define when $f$ properly represents $d  \bmod n$.  We can then use the following theorem of Watson to prove that $f(x_0, y_1, y_2 ,y_3) = x_0 - ay_1^2 - by_2^2 - aby_3^2$ represents all integers.

\begin{theorem}[Theorem 53 of Watson \cite{watson}]\label{watsonthm}
Let $d$ be an integer, let $c_0,c_1,c_2,c_3$ be non-zero integers, not all the same sign, and let 
\[f(x_0, y_1, y_2 ,y_3) = c_0x_0^2 + c_1y_1^2 + c_2y_2^2 + c_3y_3^2.\]
Then if $f$ properly represents $d \bmod 2^4|c_0c_1c_2c_3|$, then $f$ represents $d$.
\end{theorem} 

In our case, $c_0 = 1$, $c_1 = -a$, $c_2 = -b$, and $c_3 = -ab$, so we have $2^4|c_0c_1c_2c_3| = 2^4a^2b^2$.  Recall that our hypotheses state that $a$ and $b$ are square-free and $\gcd(a,b) \leq 2$.
Our method will therefore be to show that for all square-free integers $d$ and all primes $p$ dividing $ab$, $f$ properly represents $d \bmod p^2$, $f$ properly represents $d$ modulo an appropriate power of 2, and that we can ``glue'' these representations together such that $f$ properly represents $d \bmod 2^4a^2b^2$.  Section \ref{mainthm} then proves Theorem \ref{main}, and Section \ref{density} calculates the density of pairs of integers in $\N \times \N$ covered by Theorem \ref{main}.
\section{Universality of $f$}
\label{modstuff}

\subsection*{Representing integers mod $p^2$}

Throughout this and the following sections, we fix square-free $a,b \in \N$ with $\gcd(a,b) \leq 2$, and fix 
\[f(x_0, y_1, y_2, y_3) = x_0^2 - ay_1^2 - by_2^2 - aby_3^2.\]

We want to show that $f$ is {\em universal}, that is, that $f$ represents all integers.  Note that it suffices to show that $f$ represents all square-free integers, since for any $m \in \Z$, 
\[f(mx_0, my_1, my_2, my_3) = m^2(x_0^2 - ay_1^2 - by_2^2 - aby_3^2).\]

We will need a basic version of Hensel's Lemma to prove that $f$ represents all square-free integers mod $p^2$, for odd primes $p$ dividing $ab$.

\begin{lemma}[Hensel's Lemma]
Given $h(x) \in \Z[x]$, $c \in \Z$, and a prime $p$, if $h(c) \equiv 0 \bmod p$ and $h'(c) \not \equiv 0 \bmod p$, then there exists $\gamma \in \Z$ such that $\gamma \equiv c \bmod p$ and $h(\gamma) \equiv 0 \bmod p^{2}$.
\end{lemma}

Given Theorem \ref{watsonthm}, we start with the following theorem.

\begin{theorem} \label{modp2} Suppose $d \in \Z$ is square-free. Then for all odd primes $p$ dividing $ab$, $f$ properly represents $d \bmod p^2$.
\end{theorem} 

\begin{proof}
Consider $p$ prime $\in \N$ such that $p|b, p \nmid a$, without loss of generality.  We are trying to find $x_0, y_i$ such that $$d \equiv x_0^2 -ay_1^2 -by_2^2 -aby_3^2 \bmod p^2.$$ 

\underline{Case 1:} $d \not\equiv 0 \bmod p.$

We will first try to find $x_0, y_i$ such that $$d + ay_1^2 \equiv x_0^2 \bmod p^2.$$  
 
Note that since $v^2 \equiv (-v)^2 \bmod p$, there are (including 0) $\frac{p+1}{2}$ quadratic residues mod $p$.  Thus, $d + ay_1^2$ and $x_0^2$ each must be one of $\frac{p + 1}{2}$ distinct possible residues.
Therefore, by the Pigeonhole Principle, there must exist $\beta_0 \equiv x_0 \bmod p$ and $\beta_1 \equiv y_1 \bmod p$ such that 
\begin{equation}
    d + a\beta_1^2 \equiv \beta_0^2 \bmod p^2. \label{hen1}
\end{equation}
Now, we must consider the cases where $\beta_0 \equiv 0 \bmod p$ and where $\beta_0 \not \equiv 0 \bmod p$ in order to find our proper representation.

\par \underline{Case 1a:} $\beta_0 \not \equiv 0 \bmod p$.

Let $h(x) = x^2 - a\beta_1^2 - d$. We know that $\beta_0$ is a root of $h(x) \bmod p$, due to Equation (\ref{hen1}). Since $h'(\beta_0) \equiv 2\beta_0 \not \equiv 0 \bmod p$, by Hensel's Lemma, there exists $\delta_0 \in \Z$, such that $\delta_0 \equiv \beta_0 \bmod p$ and $h(\delta_0) \equiv 0 \bmod p^2$. Therefore, $$d \equiv \delta_0^2 - a\beta_1^2 \bmod p^2.$$ 

Notice then that $f(\delta_0, \beta_1, 0, 0) \equiv d \bmod p^2$, meaning that we have a representation of $d \bmod p^2$; however, we are not guaranteed that this is a proper representation. Since $\delta_0 \equiv \beta_0 \not \equiv 0 \bmod p$, $\delta_0^2$ will not vanish mod $p^2$. Since $-bp^2$ will vanish mod $p^2$, 
$f(\delta_0, \beta_1, 0, 0) \equiv f(\delta_0, \beta_1, p, 0) \equiv d \bmod p^2$.
Since $\gcd(\delta_0, p) = 1$, we therefore have $f(\delta_0, \beta_1, p, 0)$ as a proper representation of $d \bmod p^2$.

\par \underline{Case 1b:} $\beta_0 \equiv 0 \bmod p$. 

Let $h(x) = ax^2 + d$. We know $\beta_1$ is a root of $h(x) \bmod p$. Since $h'(\beta_1) \equiv 2a\beta_1 \not \equiv 0 \bmod p$, by Hensel's Lemma, there exists $\delta_1 \equiv \beta_1 \bmod p$ and $f(\delta_1) \equiv 0 \bmod p^2$. So, $$d \equiv - a\delta_1^2 \bmod p^2.$$ 

Note then that $\delta_1 \equiv \beta_1 \not \equiv 0 \bmod p$, since it would imply $d \equiv 0 \bmod p^2$. Notice that $f(\beta_0, \delta_1, 0, 0) \equiv f(\beta_0, \delta_1, p, 0) \equiv d \bmod p^2$. Since $\gcd(\delta_1, p) = 1, f(\beta_0, \delta_1, p, 0)$ is therefore a proper representation of $d \bmod p^2$.

\underline{Case 2:} $d \equiv 0 \bmod p$, $d \not \equiv 0 \bmod p^2$. 

We will here try to find $y_2$ and $y_3$ such that
$$d \equiv -by_2^2 -aby_3^2 \bmod p^2.$$

Since $b \equiv d \equiv 0 \bmod p$, there exist $\widehat{b}, \widehat{d} \in \Z$ that are not $0 \bmod p$ such that $p\widehat{b} = b$ and $p\widehat{d} = d$. Since $b$ and $d$ are square-free, we have that $\widehat{b}, \widehat{d} \not \equiv 0 \bmod p$.
Then $$\widehat{d} \equiv -\widehat{b}y_2^2 -a\widehat{b}y_3^2 \bmod p.$$ 
Since $\widehat{b} \not \equiv 0 \bmod p$ and $p$ is prime, $\widehat{b}$ has an inverse mod $p$.
Multiplying by this inverse, 
\begin{align*}
    {\widehat{b}}^{-1} \widehat{d} &\equiv -y_2^2 -ay_3^2 \bmod p \\
    -(\widehat{b}^{-1}\widehat{d} + ay_3^2) &\equiv y_2^2 \bmod p
\end{align*}

Recall that there are $\frac{p + 1}{2}$ possible residues for $-(\widehat{b}^{-1}\widehat{d} + ay_3^2)$ and $y_2^2$.
Therefore, by the Pigeonhole Principle, there must exist $\beta_2 \equiv y_2 \bmod p$ and $\beta_3 \equiv y_3 \bmod p$ such that 
\begin{align*}
    -(\widehat{b}^{-1}\widehat{d} + a\beta_3^2) &\equiv \beta_2^2 \bmod p \\
    \widehat{d} &\equiv -\widehat{b}\beta_2^2 -a\widehat{b}\beta_3^2 \bmod p \\
    d &\equiv -b\beta_2^2 -ab\beta_3^2 \bmod p^2
\end{align*}
Notice that $f(0, 0, \beta_2, \beta_3) \equiv d \bmod p^2$. This is not guaranteed to be a proper representation. Since $d \not \equiv 0 \bmod p^2$ by the declaration of the case, at least one of $\beta_2, \beta_3 \not \equiv 0 \bmod p$. Thus, $\gcd(\beta_2, \beta_3, p) = 1$. So, $f(0, 0, \beta_2, \beta_3) \equiv f(0, p, \beta_2, \beta_3) \equiv d \bmod p^2$, ensuring a proper representation.

\end{proof}

\subsection*{Representing integers mod $2^{\nu}$} \label{mod2}

For this section, let $\nu$ be the highest power of 2 dividing $2^4a^2b^2$. (So $2^{\nu} = 16, 64$, or 256, since $a$ and $b$ are square-free.)

\begin{lemma}
Suppose $d \not \equiv 0 \bmod 4$.  Then $f$ properly represents $d \bmod 2^{\nu}$.
\end{lemma}

\begin{proof}
We will repeatedly use the following fact (from a version of Hensel's Lemma): 
\begin{equation*}
    \text{If } \nu \geq 3, \text{ and }m \in \Z, \text{ then }8m + 1\text{ is a quadratic residue mod }2^{\nu}.  \tag{$\ddagger$}
\end{equation*}
The benefit of this fact is that if (for example) there exists a solution to 
\begin{equation*} \label{repmod2t}
    x_0^2 - ay_1^2 - by_2^2 - aby_3^2 \equiv d \bmod 8.
\end{equation*}
with $x_0$ odd and $\gcd(x_0,y_1,y_2,y_3) = 1$, then $x_0^2 \equiv 1 \bmod 8$, and there exists $m$ such that
\begin{equation*}
     (8m+1) - ay_1^2 - by_2^2 - aby_3^2 \equiv d  \bmod 2^{\nu}.   
\end{equation*}
Since $8m+1$ is a quadratic residue, and our modulus is a power of 2 with $\nu \geq 3$, there must exist $\widehat{x_0} \in \Z$ such that $\gcd(\widehat{x_0},y_1,y_2,y_3) = 1$ and 
\begin{equation*}
     \widehat{x_0}^2 - ay_1^2 - by_2^2 - aby_3^2 \equiv d  \bmod 2^{\nu}.
\end{equation*}

We then have three cases: when $a$ and $b$ are both odd, when exactly one of $a$ and $b$ are odd, and when both $a$ and $b$ are even.

\underline{Case 1:} Suppose that $a$ and $b$ are both odd, so $2^{\nu} = 16$; we are then trying to find solutions to 
\begin{equation} \label{repmod16}
    x_0^2 - ay_1^2 - by_2^2 - aby_3^2 \equiv d \bmod 16.
\end{equation}
Since $d \not \equiv 0 \bmod 4$, at least one of $x_0$ or one of the $y_i$s must be odd; since $a,b$ and $ab$ are odd, it does not matter which of the variables is odd. So by ($\ddagger$) and the method above, we only need to solve Equation (\ref{repmod16}) mod 8.

First, note that exactly one or three of $1, -a, -b$, and $-ab$ will be $1 \bmod 4$.  (Similarly three or one of them will be $-1 \bmod 4$).  So, since three of $1, -a, -b$ and $-ab$ must have the same residue mod 4 (either 1 or $-1$), we can use those three terms (out of $x_0^2$, $-ay_1^2$, $-by_2^2$, and $-aby_3^2$) to represent $d \bmod 4$, and let the other square be 0 or 4 to ensure that we represent $d \bmod 8$.  And finally, we use ($\ddagger$) as necessary to represent $d \bmod 16$.

To ensure that we are properly representing $d$, suppose that from above we have that $f(x_0,y_1,y_2,y_3) \equiv d \bmod 16$; note that in our construction above at least one of the variables must be odd.  Therefore $\gcd(x_0,y_1,y_2,y_3)$ must be odd.  However, we also chose in our construction one of the variables such that its square was 0 or 4; note that we can choose this variable to be a power of $2$ (since $4^2 \equiv 0 \bmod 16$ and $2^2 \equiv 4 \bmod 16$).  So if $\gcd(x_0,y_1,y_2,y_3)$ is odd, it must equal 1, thus we have that we can properly represent $d \bmod 16$.

\underline{Case 2:} Suppose, without loss of generality, that $a$ is odd and $b \equiv 2 \bmod 4$; we are now trying to find solutions to 
\begin{equation*} \label{repmod64}
    x_0^2 - ay_1^2 - by_2^2 - aby_3^2 \equiv d \bmod 64.
\end{equation*}
Note that we have $-b \equiv -ab \equiv 2 \bmod 4$.  Then, since $d \not\equiv 0 \bmod 4$, if $-a \equiv 1 \bmod 4$ we can represent $d \bmod 8$ with $x_0$ odd, allowing us to use ($\ddagger$) to represent $d \bmod 64$.

If $-a \equiv 3 \bmod 4$, the only problematic case is when $d \equiv 2 \bmod 4$.  But in this case we can represent $d \bmod 16$ with $y_2$ odd and $x_0,y_3 \in \{0,2\}$.  Since in this case $b$ and $d$ are even, representing $d \bmod 16$ is sufficient to allow us to use ($\ddagger$) to represent $d \bmod 64$.

\underline{Case 3:} Suppose that $a \equiv b \equiv 2 \bmod 4$; we are now trying to find solutions to 
\begin{equation} \label{repmod256}
    x_0^2 - ay_1^2 - by_2^2 - aby_3^2 \equiv d \bmod 256.
\end{equation}
In this case, we have $-a \equiv -b \equiv 2 \bmod 4$, and $-ab \equiv 4 \bmod 8$.  If $d$ is odd, then by ($\ddagger$) we only have to represent $d \bmod 8$; choosing $y_1, y_3 \in \{0,1\}$ lets us represent every odd residue class.  The argument in the last paragraph of Case 1 similarly ensures that we can properly represent $d \bmod 256$.

If $d$ is even, then any solution to Equation (\ref{repmod256}) must have $x_0$ even; we are then trying to find solutions to 
\begin{equation*} \label{repmod128}
    \frac{1}{2}x_0^2 - \frac{1}{2}ay_1^2 - \frac{1}{2}by_2^2 - \frac{1}{2}aby_3^2 \equiv d \bmod 128.
\end{equation*}
But this is the same situation (2 odd coefficients, 2 coefficients equivalent to $2 \bmod 4$) as Case 2, as ($\ddagger$) again allows us to reduce the problem to finding a proper representation mod 8.
\end{proof}

We are now ready to glue these representations together and apply Watson's Theorem.

\subsection*{Representing integers mod $2^4a^2b^2$} \label{crt}

The following lemma ensures that if $f$ properly represents $d$ modulo each prime power dividing $2^4a^2b^2$, then $f$ properly represents $d \bmod 2^4a^2b^2$.

\begin{lemma} \label{gluing}
Take $c \in \Z$, $m,n \in \N$ and let $f(x_0,y_1,y_2,y_3) = x_0^2 -ay_1^2 - by_2^2 -aby_3^2$.  Suppose for $i \in \{0,1,2,3\}$ that $\alpha_i, \beta_i$ are integers such that 
\begin{itemize}
	\item $f(\alpha_0,\alpha_1,\alpha_2,\alpha_3) \equiv c \bmod m$, 
	\item $f(\beta_0,\beta_1,\beta_2,\beta_3) \equiv c \bmod n$, and 
	\item $\gcd(\alpha_0,\alpha_1,\alpha_2,\alpha_3) = \gcd(\beta_0,\beta_1,\beta_2,\beta_3) = 1$.  
\end{itemize}
Then there exists integers $\gamma_0,\gamma_1,\gamma_2,\gamma_3$ such that ${f(\gamma_0,\gamma_1,\gamma_2,\gamma_3) \equiv c \bmod mn}$ and $\gcd(\gamma_0,\gamma_1,\gamma_2,\gamma_3) = 1$.
\end{lemma}

\begin{proof}
By the Chinese Remainder Theorem we know that we can combine the $\alpha_i$s and $\beta_i$s to represent $c \bmod mn$. Notice that while we are guaranteed a representation of $c \bmod mn,$ the Chinese Remainder Theorem does not guarantee a \textit{proper} representation.

Let $\gamma_i = nt\alpha_i + ms\beta_i$ where $s, t \in \Z$ such that $ms \equiv 1 \bmod n$ and $nt \equiv 1 \bmod m$. Note that $f(\gamma_0,\gamma_1,\gamma_2,\gamma_3) \equiv f(\alpha_0,\alpha_1,\alpha_2,\alpha_3) \bmod m$ and $f(\gamma_0,\gamma_1,\gamma_2,\gamma_3) \equiv f(\beta_0,\beta_1,\beta_2,\beta_3) \bmod n$. We then have $f(\gamma_0,\gamma_1,\gamma_2,\gamma_3) \equiv c \bmod mn$ by the Chinese Remainder Theorem; we will show that we can use the $\gamma_i$s to get a proper representation.

Let $D = \gcd(\gamma_0, \gamma_1, \gamma_2, \gamma_3)$. Notice that if $D = 1$, then we have a proper representation, so assume $D > 1$.

Suppose $\gcd(D, mn) > 1$.  Without loss of generality, we can assume that $\gcd(D, m) > 1$. Since $\gcd(D, m) > 1$, there exists some prime $q$ such that $q | m$ and $q | D$. Recall that $\gamma_i = ms\alpha_i + nt\beta_i$. Since $q | m$, $ms\beta_i$ vanishes mod $q$. Since $q|m$ and $nt \equiv 1 \bmod m$, $nt \equiv 1 \bmod q$. Thus, $\gamma_i \equiv \alpha_i \bmod q$. Since $q | D$, we have $q | \gamma_i$. Therefore, $q | \alpha_i$ which means that $\gcd(\alpha_0,\alpha_1,\alpha_2,\alpha_3) > 1$, which contradicts the hypothesis that $\gcd(\alpha_0,\alpha_1,\alpha_2,\alpha_3) = 1$. So, $\gcd(D, mn) = 1$.
Let $F = \gcd(\gamma_1, \gamma_2, \gamma_3)$, and let $E$ be the largest factor of $F$ coprime to $D$. Then by definition, $\gcd(D, E) = 1$ and $\gcd(\gamma_0, E) = 1$. Let $\widehat{\gamma_0} = \gamma_0 + Emn$. Since $\widehat{\gamma_0}^2 \equiv \gamma_0^2 \bmod mn$, we have that $f(\widehat{\gamma_0}, \gamma_1, \gamma_2, \gamma_3) \equiv c \bmod mn$. 

We then need to show that this is a proper representation of $c$.  Note that $\gcd(D,mn) = \gcd(\gamma_0, F, mn) = 1$.  Suppose then that $p$ is a prime such that  $p|\gcd(\widehat{\gamma_0},F)$.  Then $p|\gamma_0 + Emn$.  First, suppose that $p|E$.  Then since $p|E$ and $p|\gamma_0 + Emn$, we have that $p|\gamma_0$.  So, since $p|F$ and $p|\gamma_0$, we have that $p|D$.  But then $p|D$ and $p|E$, which contradicts the definition of $E$.

So $p\nmid E$.  But then since $p|F$ and $p\nmid E$, the definition of $E$ implies that $p|D$.  But then $p|\gamma_0$, and since $p|\widehat{\gamma_0} = \gamma_0 + Emn$, we have that $p|Emn$.  We assumed that $p \nmid E$, so we must have $p|mn$.  But then $p|\gcd(D,mn)=1$, a contradiction.  So $\gcd(\widehat{\gamma_0},F) = \gcd(\widehat{\gamma_0}, \gamma_1, \gamma_2, \gamma_3) = 1$, and we have a proper representation of $c \bmod mn$.
\end{proof}

\section{Proof of Main Theorem} \label{mainthm}

We can now assemble the results of the previous sections to prove Theorem \ref{main}.

\begin{proof}[Proof of Theorem \ref{main}]
Let $a$ and $b$ be positive square-free integers such that $\gcd(a,b) \leq 2$.  By Lemma 2.3 of Cooke et al. \cite{chw}, there are elements of $LQ_{a,b}$ that cannot be written as the sum of two squares, so $g_{a,b} (2) \geq 3$.

Then, take $\alpha = \alpha_0+2\alpha_1\bi+2\alpha_2\bj+2\alpha_3\bk \in LQ_{a,b}^2$; we need to prove that there exist $x,y,z \in LQ_{a,b}$ such that $\alpha = x^2 + y^2 + z^2$.  As noted at the beginning of Section \ref{2}, if we let $z = 1 +\alpha_1\bi +\alpha_2\bj +\alpha_3\bk$, then $\alpha - z^2 \in \Z$.  Letting $x = x_0 \in \Z$ and $y = y_1\bi + y_2 \bj + y_3 \bk$, it then suffices to then show that $f(x_0,y_1,y_2,y_3) = x_0^2 - ay_1^2 - by_2^2 - aby_3^2$ represents all integers.

By the note at the beginning of Section \ref{modstuff}, it suffices to show the $f$ represents all square-free integers.  By Theorem \ref{modp2}, we know that $f$ properly represents all square-free integers mod $p^2$ for all odd primes $p$ dividing $ab$, and by Lemma \ref{mod2}, we know that $f$ properly represents mod $2^t$ all integers not equivalent to $0 \bmod 4$, where $2^t$ is the highest power of 2 dividing $2^4a^2b^2$.  Lemma \ref{gluing} then implies that $f$ represents all square-free integers properly mod $2^4a^2b^2$, and by Watson's Theorem (Theorem \ref{watsonthm}) we have that $f$ represents all integers.  Therefore there exist $x,y \in LQ_{a,b}$ such that $\alpha - z^2 = x^2 + y^2$, so $g_{a,b} (2) \leq 3$. So, $g_{a,b} (2) = 3$, thus completing the proof. 

\end{proof}

\section{Density results} \label{density}

Let $C(x) = \{(a,b) \in \N \times \N \mid a,b \leq x; \, a,b \text{ square-free; and }\gcd(a,b) = 1\}$; we call such pairs $(a,b)$ {\em strongly carefree couples}.  Theorem 1 of Moree \cite{moree} states that
\[C(x) = \frac{x^2}{\zeta(2)^2}\prod_p \left(1 - \frac{1}{(p+1)^2}\right) + O(x^{3/2}),\] 
so the probability of a pair of integers $(a,b)$ to be a strongly carefree couple is
\[K = \frac{1}{\zeta(2)^2}\prod_p \left(1 - \frac{1}{(p+1)^2}\right) \approx 0.286747.\]

\begin{proof}[Proof of Corollary \ref{densitythm}]
First, we let $C_{ev}(x) = \{(a,b) \in C(x) \mid ab \text{ even}\}$ and $C_{od}(x) = \{(a,b) \in C(x) \mid ab \text{ odd}\}$.  We clearly have that $C(x)$ is the disjoint union of $C_{ev}(x)$ and $C_{od} (x)$.  Additionally, note that $(a,b) \in C_{ev}(x)$ if and only if exactly one of $(a,b/2)$ and $(a/2,b)$ is in $C_{od}(x)$; since these pairs lie in range with area each half that of $C_{ev}(x)$, we get that $C_{ev} \sim C_{od}$.

We then let $CC(x) = \{(a,b) \in \N \times \N \mid a,b \leq x; \, a,b \text{ square-free; }\gcd(a,b) \leq 2 \}$.  Note that $C(x) \subseteq CC(x)$, and that 
\begin{align*}
    CC(x) - C(x) & = \{(a,b) \in \N \times \N \mid a,b \leq x; \, a,b \text{ square-free; and}\gcd(a,b) = 2 \} \\
    & = \{(2a_0,2b_0) \in \N \times \N \mid (2a_0,2b_0) \in C(x), a_0,b_0 \text{ odd} \} \\
    & = \{(2a_0,2b_0) \in \N \times \N \mid (a_0,b_0) \in C_{od}(x/2)\}
\end{align*}
Then, since $C_{od} \sim C/2$ and $C(x/2) \sim C(x)/4$, we get that $(CC - C) \sim C/8$; therefore, the probability that a pair of positive integers $(a,b)$ satisfies the hypothesis of Theorem \ref{mainthm} -- that $a$ and $b$ are square-free and $\gcd(a,b) \leq 2$ -- is
\[\frac{9K}{8} \approx 0.322590.\]
\end{proof}

\section{Open questions}

Theorem 3.6 of Cooke et al. \cite{chw} states that if $a \equiv b \equiv 0 \bmod 4$, then $g_{a,b}(2) = 5$, so $D(5,2) \geq 0.0625.$  This and Corollary \ref{densitythm} are the only positive density results we know of at this time.  The methods here can be expanded to a few other cases, but other methods will likely have to be used to get the known cases over 50\%.  

In cases where $a$ and $b$ are odd, but $\gcd(a,b)>1$, computational evidence seems to suggest that $g_{a,b}(2) = 3$, but these cases have modular obstructions to the methods used in this paper.  We have tried a number of different methods to show that $g_{3,3}(2) = 3$, without success.

It should also be noted that throughout this paper, we assume that $a$ and $b$ are positive.  Given that Watson's Theorem applies to any indefinite quadratic form, and has only modular restrictions, we expect that results when $a$ and $b$ have opposite signs would look very similar to the results we have here.

\bibliographystyle{plainnat}

\end{document}